\documentclass[12pt]{amsart}

\usepackage{amsmath, amssymb, amsfonts, amsbsy, amsthm, color}
\usepackage{graphicx}
\usepackage[table]{xcolor}
\usepackage{tikz}
\usepackage{multirow}
\usepackage{subcaption}

\usepackage[margin=1.2in]{geometry}
\everymath{\displaystyle}

\newtheorem{theorem}{Theorem}[section]

\newtheorem{lemma}[theorem]{Lemma}
\newtheorem{corollary}[theorem]{Corollary}

\theoremstyle{remark}
\newtheorem{remark}[theorem]{Remark}

\numberwithin{equation}{section}

\newcommand{\C}{{\mathbb{C}}}
\newcommand{\sd}{\sigma_{d-1}}
\newcommand{\Ec}{{\mathcal{E}}}
\newcommand{\Dc}{{\mathcal{D}}}

\newcommand{\Sc}{{\mathcal{S}}}

\newcommand{\Mcc}{{\mathcal{M}}}
\newcommand{\R}{{\mathbb{R}}}
\newcommand{\Hb}{{\mathbb{H}}}
\newcommand{\Sb}{\mathbb{S}}
\newcommand{\Sbd}{\mathbb{S}^{d-1}}

\newcommand{\E}{{\mathbb{E}}}

\newcommand{\Pb}{{\mathbb{P}}}
\newcommand{\Mb}{{\mathbb{M}}}
\newcommand{\Sbb}{{\mathbb{S}}}
\newcommand{\trace}{{\rm{Trace}}}
\newcommand{\sym}{{\rm{sym}}}
\newcommand{\proj}{{\rm{proj}}}

\begin{document}

\title{On the Search for Tight Frames of Low Coherence}

\author{Xuemei Chen$^\dagger$}
\address{Department of Mathematics and Statistics,
University of North Carolina Wilmington}
\email{chenxuemei@uncw.edu}
\thanks{\noindent $^\dagger$ The research of this author was supported
by the U. S. National Science Foundation under grant DMS-1908880.}
\author{Douglas P. Hardin$^*$}
\address{Center for Constructive Approximation, Department of Mathematics,
Vanderbilt University, Nashville, Tennessee 37240}
\email{doug.hardin@Vanderbilt.Edu}

\author{Edward B. Saff$^*$}
\address{Center for Constructive Approximation, Department of Mathematics,
Vanderbilt University, Nashville, Tennessee 37240}
\email{edward.b.saff@vanderbilt.edu}
\thanks{\noindent \ \ \ $^*$ The research of these authors was supported, in part,
by the U. S. National Science Foundation under grant DMS-1516400.
}
 

\date{\today}


\begin{abstract}
We introduce a projective Riesz $s$-kernel for the unit sphere
$\mathbb{S}^{d-1}$ and investigate properties
of  $N$-point energy minimizing configurations for such a kernel.  We show
that these configurations, for $s$ and $N$ sufficiently large, form frames that are well-separated (have low coherence) and are nearly tight. Our results
suggest an algorithm for computing well-separated tight frames which
is illustrated with numerical examples.

\end{abstract}

\maketitle

%
%

\section{Introduction}
%
%
%
%
%
%
%
%

A set of vectors $X=\{x_i\}_{i\in I}$ is a \emph{frame}\footnote{Depending on the context, we either consider $X$ to be a multiset, allowing for repetition, or as an ordered list.} for a separable Hilbert space $H$ if there exist $A,B>0$ such that for every $x\in H$,
$$A\|x\|^2\leq\sum_{i\in I}|\langle x,x_i\rangle|^2\leq B\|x\|^2.$$
The constant $A$ ($B$, resp.) is called the \emph{lower (upper, resp.) frame bound}.
When $A=B$, $X$ is called a \emph{tight frame}, which generalizes the concept of an orthonormal basis in the sense that the recovery formula $x=\frac{1}{A}\sum_{i\in I}\langle x,x_i\rangle x_i$ holds for every $x\in H$. 
For the finite dimensional space $H=\Hb^d$, where $\Hb=\R$ or $\C$, $X=\{x_i\}_{i=1}^N$ is a frame of $\Hb^d$ if and only if $\{x_i\}_{i=1}^N$  spans $\Hb^d$. We shall also use $X$ to denote the matrix whose $i$th column is $x_i$ and therefore we have
$$X \text{ is tight with frame bound } A \Longleftrightarrow XX^*=AI_d,$$
where $I_d$ is the $d\times d$ identity matrix.


Let $\Sc(d,N):=\{X=\{x_i\}_{i=1}^N\subset\Hb^d:  \|x_i\|=1\}$ be the collection of all $N$-point configurations on $\Sb^{d-1}$, the unit sphere of $\Hb^d$, where $\|\cdot\|$ denotes the $\ell_2$ norm. If we have a unit norm tight frame $X\in\Sc(d,N)$, then it is well known that the frame bound has to be $N/d$ since
\begin{equation}\label{equ:tight}
XX^*=\frac{N}{d}I_d.
\end{equation}

Benedetto and Fickus   show in~\cite{BF03} that frames that attain
\begin{equation}\label{equ:fp}
\min_{X\in\Sc(d,N)} \sum_{i\neq j}|\langle x_i,x_j\rangle|^2
\end{equation}
 are precisely the unit norm tight frames.
We will call the function $|\langle x, y\rangle|^2$ the \emph{frame potential kernel}. Ehler and Okoudjou \cite{EO12} generalized this result to the  $p$-frame potential kernel $|\langle x,y\rangle|^p$, see also \cite{BGMPV19} for recent results on $p$-frames.

Separation is a desirable property of a frame. It is quantified by the \emph{coherence} $\xi(X)$ of a frame  defined for $X\in\Sc(d,N)$ by
$$\xi(X):= \max_{i\neq j}|\langle x_i,x_j\rangle|.$$
The smaller the coherence, the better separated the frame is. 

A straightforward method to find well-separated frames is to solve
\begin{equation}\label{equ:sep}
\xi_{N}:=\min_{X\in\Sc(d,N)}\xi(X)=\min_{X\in\Sc(d,N)} \max_{i\neq j}|\langle x_i,x_j\rangle|,
\end{equation}
which has been studied in several works including Welch \cite{W74}, Conway et al. \cite{CHS96}, Strohmer and Heath~\cite{SH03}, and more recently \cite{FJM17, BH14}. 
The problem \eqref{equ:sep} is often referred as the \emph{best line-packing} problem because it asks how to arrange $N$ lines in $\Hb^d$ so that they are as far apart as possible. Conway et al~\cite{CHS96} made extensive computations on this problem from a more general perspective: how to best pack $n$-dimensional subspaces in $\R^m$? There are also many other contributions using tools in geometry and combinatorics \cite{DGS91},  and statistics \cite{BEG17}.
A minimizer of \eqref{equ:sep} is called a \emph{Grassmannian frame} by~\cite{SH03} and we shall use this terminology as well.
Our goal in this paper is to develop methods for generating tight frames with low coherence using   energy minimization   on the projective space  $\Hb \Pb^{d-1}$, which consists of all lines  in $\Hb^d$ through the origin;   
namely,    sets of the form 
\begin{equation}\label{elldef}\ell(x):=\{\alpha x:\alpha\in \Hb\}, \end{equation}
for some $x\in \Sbd$. We  endow   $\Hb \Pb^{d-1}$ with the metric   
\begin{equation}
\rho(\ell(x),\ell(y)):= \sqrt{ 2-2|\langle x,y\rangle|^2}, \qquad x,y\in \Sb^{d-1},
\end{equation}
which is   the `chordal' distance\footnote{The chordal distance between the lines $\ell(x)$ and $\ell(y)$ is given by $\min\{\|x-uy\|\colon u\in \mathbb{H}, |u|=1\}$.}
and   utilize  kernels  on $\Sbd\times\Sbd$ of the form:
\begin{equation}\label{equ:Kf}
K(x,y)= f\left(\rho(\ell(x),\ell(y))\right)=f\left(\sqrt{ 2-2|\langle x,y\rangle|^2}\right), \quad x,y\in \Sbd\subset\Hb^d.
\end{equation} 
 The energy of  $X=X_N=\{x_1,\ldots, x_N\}$   with respect to the kernel $K$  is given by
\begin{equation}\label{equ:EK}
E_K(X_N):=\sum_{i\neq j}K(x_i,x_j).
\end{equation}
One seeks   the infimum of \eqref{equ:EK} over all possible $N$ point configurations on $\Sbd$.  Assuming  $K$ is lower semi-continuous  on $\Sbd\times\Sbd$   so the infimum is attained,   we define the \emph{$N$-point minimal energy of kernel $K$}   as
\begin{equation}\label{equ:disg}
\mathcal{E}_K(\Sbd,N):=\min_{X\in \Sc(d,N)}E_K(X).
\end{equation}
  An  $N$-point configuration that achieves the minimum \eqref{equ:disg} will be denoted by $X_N^*(K, \Sbd)$ (or $X_N^*$ when there is no ambiguity).  
  So far the minimal energy and optimal configuration have been confined to the sphere and generalizes to any compact set $A$, and will be denoted as $\mathcal{E}_K(A,N), X_N^*(K, A)$ respectively.
  Note that the frame potential $|\langle x,y \rangle|^2$ is of the form \eqref{equ:Kf} and that the   energy minimizers are precisely the unit norm tight frames.   However, in general for $N>d$, these minimizers may not consist of well-separated lines.  

To achieve well separation of lines our approach is to first   consider a class of kernels that are   strongly repulsive and analyze the approximate tightness of their energy minimizers relative to their frame potential energy.       
Specifically, we introduce the \emph{Riesz projective $s$-kernel}
\renewcommand{\arraystretch}{2.2}
\begin{equation}\label{equ:proj}
G_s(x,y):=\left\{\begin{array}{ll}\log\dfrac{1}{1-|\langle x,y\rangle|^2}, & s=0\\\dfrac{1}{(1-|\langle x,y\rangle|^2)^{s/2}}, & s> 0 \end{array}\right.
\end{equation} for $x,y\in \Sbd$ and seek solutions to the problem
\begin{equation}\label{equ:projective}
\min_{X\in\Sc(d,N)} \sum_{i\neq j}G_s(x_i,x_j).
\end{equation}  
The kernel   $G_s$ is a modification of the classical \emph{Riesz $s$-kernel}
defined for $x,y$ in a normed linear space $(V,\|\cdot\|)$ as
\begin{equation}\label{equ:R}
R_s(x,y)=\left\{\begin{array}{ll}\log\frac{1}{\|x-y\|}, & s=0\\\frac{1}{\|x-y\|^s}, & s> 0. \end{array}\right. 
\end{equation}
In fact, as we will show in \eqref{equ:red}, the projective Riesz $s$-kernel $G_s$ can also be represented in terms of $R_s$ for an appropriate subspace $V$ of matrices with the Frobenius norm.  


Notice that minimizers of \eqref{equ:projective} will avoid antipodal points since   the energy in that case would be infinite. 
The connection between projective Riesz $s$-kernels and Riesz $s$-kernels is more immediate in the real case $\R^d$; since $\|x\pm y\|^2=2\pm2\langle x,y\rangle$, we have
$$\frac{1}{(1-|\langle x,y\rangle|^2)^{s/2}}=\frac{2^s}{\|x-y\|^s\|x+y\|^s}.$$ Thus the projective Riesz kernel is just the Riesz kernel with the multiplicative factor $\|x+y\|^s$ to account for antipodal points.

A major focus of this paper is to exploit connections between $G_s$ and $R_s$ and 
reduce solving the projective Riesz $s$-kernel minimization problem \eqref{equ:projective} to solving \begin{equation}\label{equ:riesz}
\min_{X\subset \Dc,|X|=N} \sum_{i\neq j}R_s(x_i,x_j),
\end{equation}
where we take $\Dc$ to be the projective space, but embedded in a higher dimensional real vector space (see Section \ref{sec:embed}).
There are well established theorems available in the minimal energy literature for Riesz $s$-kernels (see e.g. \cite{BHS20}) and we shall review some of them in Section \ref{sec:pre}. 

The projective Riesz $s$-kernel  for   $s<0$ defined by $G_s(x,y):=-(1-|\langle x,y\rangle|^2)^{-s/2}$ is also interesting. For such $s$ we will be solving
\begin{equation}\label{equ:projective:negative}
\min_{X\in\Sc(d,N)} \sum_{i\neq j}-(1-|\langle x ,y\rangle|^2)^{-s/2}\qquad\qquad(s<0).
\end{equation} 
This coincides with \eqref{equ:fp} when $s=-2$. This paper shall focus on the $s\geq0$ case in the analysis, but our numerical experiments will include optimal configurations of \eqref{equ:projective:negative}.


The contributions of this paper are two-fold: 

\begin{itemize}
\item  Derive minimal energy results for the projective Riesz kernel. See Theorems~\ref{thm:contP},  \ref{thm:circleP}, and \ref{thm:eat}.



\item  Develop methods for constructing nearly tight and well-separated frames.  See Theorems~\ref{thm:sepp} and \ref{thm:tight} and numerical results in Section \ref{sec:num:both}.
\end{itemize}



%

\section{Minimal Energy Background}\label{sec:pre}

In this section we will introduce some necessary background on minimizing discrete energy and its relation to the continuous energy. 

The discrete minimal energy problem is known to be challenging, and we  have very limited knowledge about the optimal configuration even for the classical Riesz kernel case \eqref{equ:riesz} on the 2-dimensional sphere. The following theorem settles the case when points are on a circle of a real vector space for a large class of kernels that includes  Riesz  kernels.

\begin{theorem}[Fejes-T\'oth, \cite{F56}]\label{thm:circleR}
If $r>0$  and  $f:(0,2r]\rightarrow\R$ is a non-increasing convex function defined at 0 by the (possibly infinite) value $\lim_{t\rightarrow0^+}f(t)$, then any $N$ equally spaced points on a circle of radius $r$ (in $\R^m$) minimizes the discrete energy $E_K(X_N)$ for the kernel $K(x,y)=f(\|x-y\|)$. If in addition, $f$ is strictly convex, then no other $N$-point configuration on this circle is optimal.
\end{theorem}


Less is known regarding optimal configurations for \eqref{equ:riesz} beyond $\Sb^1$. For optimality of configurations with particular cardinality $N$ and dimension $d>1$, see  \cite{CK07} and \cite{BHS20}. 
On the other hand, many asymptotic results (as $N\to \infty$) for optimal configurations on the sphere as well as on $\R^d$ are known
(for examples of recent results, see \cite{BS18}, \cite{HLSS2018}).  

For a set of $N$ points $X=\{x_i\}_{i=1}^N$, the \emph{separation distance} of $X$ is defined as
$$\delta(X):=\min_{i\neq j} \|x_i-x_j\|.$$
The \emph{best-packing problem} is to find the $N$-point configuration on $A$ that maximizes the separation distance:
\begin{equation}\label{equ:bestpacking}
\delta_N(A)=\max_{X\subset A,\, |X|=N}\delta(X)= \max_{X\subset A, \, |X|=N}\min_{i\neq j}\|x_i-x_j\|.
\end{equation}


It is immediate, for example, that the best $N$-point  packing of  $\Sbb^1\subset\R^2$ consists of  $N$ equally spaced points on the circle.  



When $s\rightarrow\infty$,  the minimization problem with respect to the Riesz kernel $R_s(x,y)$
$$\min_{X\subset A,|X|=N} E_{R_s}(X)$$ 
 turns into the best-packing problem \eqref{equ:bestpacking}; more precisely, 

\begin{theorem}[\cite{BHS20}]\label{thm:sinf}
If $N\geq2$ and $A\subset\R^m$ is  a  compact set of cardinality at least $N$, then
$$\lim_{s\rightarrow\infty}\Ec_{R_s}(A, N)^{1/s}=1/\delta_N(A),$$
where $R_s$ is the Riesz kernel defined in \eqref{equ:R}.
Furthermore, if $X_s$ is an optimal configuration that achieves $\Ec_{R_s}(A,N)$, then every cluster point as $s\rightarrow\infty$ of the set $\{X_s\}_{s>0}$ on $A$ is an $N$-point  best-packing configuration on $A$.
\end{theorem}


This discrete minimal energy problem is related to a continuous energy problem as we next describe. Let $\mathcal{M}(A)$ be the set of probability measures supported on $A$. For a general kernel $K$, the \emph{potential function} of a measure $\mu\in\mathcal{M}(A)$ with respect to $K$ is defined as
$$U^{\mu}_K(x):=\int_AK(x,y)\ d\mu(y),$$ 
provided the integral exists as an extended real number.   The energy of $\mu$ is defined  as 
$$I_K(\mu):=\int_AU_K^{\mu}(x)\ d\mu(x)=\iint_{A\times A}K(x,y)\ d\mu(x)d\mu(y),$$ 
and the {\em Wiener constant} is
\begin{equation}\label{equ:WK}
W_K(A):=\inf_{\mu\in\Mcc(A)} I_K(\mu).
\end{equation}

Likewise this infimum can be achieved, and the probability measure that optimizes the above problem is called the \emph{$K$-equilibrium measure}. The $K$-\emph{capacity} of the set $A$ is defined by
$$\text{cap}_K(A):=\frac{1}{W_K(A)}.$$ 
A set $A$ has zero capacity means that $W_K(A)=\infty$, which makes the problem \eqref{equ:WK} trivial since every probabilistic measure generates $\infty$ energy.

We now present a classical theorem connecting the discrete minimal energy problem to the continuous one. 
Before that we introduce the weak* limit of measures.
A sequence of measures $\mu_n$ converges weak* to $\mu$ if for every continuous function $f$ on $A$,
$$\lim_{n\rightarrow\infty}\int fd\mu_n=\int f d\mu.$$
We also define $\delta_x$ to be the point mass probability measure on the point $x$. Moreover, given a finite collection of points $X$, its \emph{normalized counting measure} is defined as
$$\nu(X)=\frac{1}{|X|}\sum_{x\in X}\delta_x.$$

\begin{theorem}[{\cite{C58}, \cite[Theorem 4.2.2]{BHS20}}]\label{thm:aym}
If $K$  is a  lower semicontinuous and symmetric  kernel on $A\times A$, where $A\subset\R^m$ is an infinite compact set, then 
\begin{equation}\label{equ:aym}
\lim_{N\rightarrow\infty}\frac{\Ec_K(A,N)}{N^2}=W_K(A).
\end{equation}
Moreover, every weak* limit measure (as $N\to\infty$) of the sequence of normalized counting measures $\nu(X_N^*)$ is a $K$-equilibrium measure.
\end{theorem}

The proof of Theorem~\ref{thm:aym} for the case of a Riesz kernel can also be found in the book by Landkof~\cite[Eq. (2.3.4)]{L72}.

We now review two  important facts concerning  Riesz kernels.

\begin{theorem}[{\cite{M99}, \cite{BHS08}, \cite{BHS20}}]\label{lem:zero}
Let $A\subset \R^m$ be a compact infinite subset of an $\alpha$-dimensional $C^1$-manifold with  $A$ of positive  $\alpha$-dimensional Hausdorff measure.
\begin{enumerate}
\item[(1)] If $s\in[0,\alpha)$,  then the $R_s$-equilibrium measure on $A$ is unique. Moreover, if the potential function $U_{R_s}^{\mu}$ is constant on $A$, then $\mu$ is the $R_s$-equilibrium measure on $A$.
\item[(2)] If $s\in [\alpha,\infty)$, then $A$ has   $R_s$-capacity zero. Moreover, if $X_N^*(R_s, A)$ denotes an $R_s$-energy  optimal $N$-point configuration for $N\ge 2$, then the sequence of  normalized counting measures $\nu(X_N^*(R_s, A))$ converges to the uniform measure (normalized Hausdorff measure) on $A$  in the weak* sense as $N\rightarrow\infty$ (this is a special case of the  so-called Poppy-seed bagel theorem). 
\end{enumerate}
\end{theorem}

\section{An overview of the problem on the sphere}\label{sec:overview}

\renewcommand*{\arraystretch}{1.2}
\subsection{Projectively equivalent configurations.}
Note that for a  kernel $K$ of the form   \eqref{equ:Kf},  the energy $E_K(X)$, $X=\{x_i\}_{i=1}^N\in \mathcal{S}(d,N)$, is invariant under any of the following operations on $X$:
\begin{equation}\label{equ:equiv}
\begin{array}{rl}
\text{(i)} & \text{Apply a unitary operator (or orthogonal operator if $\Hb=\R$) on $X$ as}\\
&\text{$\{Ux_i\}_{i=1}^N$};\\
\text{(ii)} & \text{Change the sign of any $x_i$};\\
\text{(iii)} & \text{Permute $x_1, \ldots, x_N$.}
\end{array}
\end{equation}
Any configuration $Y$  obtained from  $X$ by applying these operations  is said to be {\em projectively equivalent} to $X$.   
For example,    $\{x_1, x_2, x_3, x_4\}$ is  projectively equivalent to  $\{Ux_4, Ux_2, -Ux_3, Ux_1\}$. 

 Theorem~\ref{thm:circleP} below states that the configuration of equally spaced points on the half-circle,
\begin{equation}\label{equ:h}
X_N^{(h)}:=\{e^{i\cdot 0}, e^{i\frac{\pi}{N}},e^{i\frac{2\pi}{N}},\ldots, e^{i\frac{(N-1)\pi}{N}}\}\subset\R^2,
\end{equation}
is optimal for \eqref{equ:disg} for a certain class of kernels $K$. 

\subsection{From sphere to the projective space}\label{sec:embed}

 The projective space $\Hb\Pb^{d-1}$  can be embedded isometrically into the space of $d\times d$ Hermitian matrices, denoted by $\Hb\Mb_{d\times d}^h$ as we next describe. 
Note that   $\Hb\Mb_{d\times d}^h$     is a real vector space for both $\Hb=\R$ and  $\Hb=\C$ with inner product in $\Hb\Mb_{d\times d}^h$  defined as $\langle M_1, M_2\rangle =\trace(M_1^*M_2)$. This inner product induces the Frobenius norm $\|M\|=\|M\|_F$ on $\Hb\Mb_{d\times d}^h$.   We further note that $\Hb\Mb_{d\times d}^h$ with the Frobenius norm can be identified with  the Euclidean space $\R^m$ where $m=(d^2+d)/2$ when $ \Hb=\R$ and $m=d^2$ when $ \Hb=\C$ (e.g., when $\Hb=\R$ and $M=(M_{i,j})$ we take any ordering of  the $m$  numbers $ \sqrt{2}M_{i,j}$ for $i<j$ and $M_{i,j}$ for $i=j$).

Recalling \eqref{elldef}, we define $\Psi: \Hb\Pb^{d-1}\rightarrow \Hb\Mb_{d\times d}^h$   as $\Psi(\ell(x)):=p_x$  with $p_x:=xx^*$ and
 $x\in\Sb^{d-1}$.  Clearly, $\Psi$ is well defined (i.e., independent of the choice of the representative of the line). 
We  denote the range of $\Psi$ by $\Dc$; that is, $\Dc:=\Psi(\Hb\Pb^{d-1})=\Phi (\Sbd)$ where $\Phi:=\Psi\circ\ell$.
 
For $x,y\in\Sbd$, the following well known equality (see, e.g. \cite{CKM16}) establishes that $\Psi$ is an isometry:
\begin{align}\label{equ:essential}
\rho\left(\ell(x),\ell(y)\right)^2=2-2|\langle x,y\rangle|^2=\|p_x-p_y\|_F^2=\|\Phi(x)-\Phi(y)\|_F^2.
\end{align}
It is used, for example, in works on phase retrieval, see e.g.~\cite{CS13, GKK15}. For the reader's convenience
we note that the middle equality in \eqref{equ:essential} follows using the cyclic property of the trace:
$$
\|p_x-p_y\|^2_F =2-\langle p_x, p_y\rangle-\langle p_y, p_x\rangle
=2-\trace( y^*xx^*y)- \trace(x^*yy^*x)=2-2|\langle x, y\rangle|^2,
$$
from which we also get
\begin{equation}\label{equ:inner}
\langle p_x, p_y\rangle=|\langle x,y\rangle|^2.
\end{equation}
Note that \eqref{equ:essential} shows that $\Psi$ is    an isometric embedding of $\Hb\Pb^{d-1}$ in $\Hb\Mb_{d\times d}^h$,  and so we identify  $\Hb\Pb^{d-1}$ with $\Dc$.   We remark that   $\Dc$ is a real analytic manifold whose    dimension   $\dim (\Dc)=\dim(\Hb\Pb^{d-1})$   is $d-1$ in the case $\Hb=\R$ and  $2d-2$ in the case $\Hb=\C$ (see \cite{BH14} or \cite{L03}).  


Now we are able to consider a kernel   of the form $K(x,y)= f\left(\sqrt{ 2-2|\langle x,y\rangle|^2}\right)$ on $\Sbd\times \Sbd$  as a kernel 
 $K(x,y)=\widetilde K(p_x,p_y)= f(\|p_x-p_y\|)$ on  $\Dc\times \Dc$. 
Specifically  the projective Riesz $s$-kernel (see \eqref{equ:proj}) can be  reexpressed as
\renewcommand*{\arraystretch}{2}
\begin{equation}\label{equ:red}
G_s(x,y)=\left\{\begin{array}{ll}\log\frac{2}{\|p_x-p_y\|^2_F}=\log2+2R_s(p_x,p_y), & s=0\\ \rule{0pt}{30pt} \frac{2^{s/2}}{\|p_x-p_y\|_F^{s}}=2^{s/2}R_s(p_x,p_y), & s> 0. \end{array}\right.
\end{equation} This   allows us to reformulate the minimal projective energy problem in terms of the Riesz minimal energy problem on the set $\Dc$. This technique was also employed in \cite{CP15}. 
In  the next two sections we  apply  results of Section \ref{sec:pre} for the problems
\renewcommand*{\arraystretch}{1.2}
\begin{equation}\label{equ:Rcont}
\min_{\mu\in\Mcc(\Dc)}I_{R_s}(\mu)
\end{equation}
\begin{equation}\label{equ:rieszC}
\min_{\{p_i\}_{i=1}^N\subset\Dc} \sum_{i\neq j}R_s(p_i,p_j).
\end{equation}

Similarly, the  Grassmannian problem \eqref{equ:sep} is equivalent to the best-packing problem \eqref{equ:bestpacking} on the projective space, which is to maximize the smallest pairwise distance between all the lines (frame vectors). Let $P=\{p_i\}_{i=1}^N\subset\Dc$. For any point $p_i\in \Dc$, we can find $x_i\in\Sbd$ such that $p_i=x_ix_i^*$. By \eqref{equ:essential},
\begin{equation}\label{equ:coh}
\delta^2(P)=\min_{i\neq j}\|p_i-p_j\|^2=\min_{i\neq j}\left(2-2|\langle x_i,x_j\rangle|\right)=2-2\max_{i\neq j}|\langle x_i,x_j\rangle|=2-2\xi(X).
\end{equation}
So
\begin{align}\label{equ:2}
\delta^2_N(\Dc) =\max_{\{p_i\}_{i=1}^N\subset \Dc}\delta^2(P)= \max_{\{x_i\}_{i=1}^N\subset \Sbd}(2-2\xi(X))=2-2\xi_{N}.
\end{align}
The last equality is from the definition \eqref{equ:sep}.

For any Borel probability measure $\mu$ on the sphere, this embedding also induces the pushforward (probability) measure $\mu_\proj$ on $\Dc\subset\Hb\Mb_{d\times d}^h$. 
 By definition of a pushforward measure, 
\begin{equation}\label{equ:mup}
\mu_\proj(\mathcal{B}):=\mu(\Phi^{-1}(\mathcal{B}))), \quad \text{for   Borel measurable }\mathcal{B}\subset \Dc.
\end{equation} We shall also write $\Phi(\mu)$ for $\mu_\proj$. 

To better understand $\mu_\proj$, we further consider the symmetrization $\mu_\sym$ of a measure $\mu\in\Mcc(\Sbd)$ defined as
\renewcommand{\arraystretch}{2.2}
\begin{equation}\label{equ:sym}
\mu_\sym(B)=\left\{\begin{array}{ll}\frac{\mu(B)+\mu(-B)}{2}, &\Hb=\R\\
\frac{1}{2\pi}\int_0^{2\pi}\mu(e^{i\theta}B)d\theta,&\Hb=\C\end{array}\right.,
\end{equation} for   Borel measurable   ${B}\subset\Sbd.$
\renewcommand{\arraystretch}{1}


It is not difficult to show that $\mu_\sym=\tilde\mu_\sym$ if and only if the pullback measures $\mu\circ\ell^{-1}$ and $\tilde\mu\circ\ell^{-1}$ agree.  The injectivity of $\Psi$ then shows  \begin{equation}\label{equ:mu}
\Phi(\mu)=\Phi(\tilde\mu)\Longleftrightarrow\mu_\sym=\tilde\mu_\sym.
\end{equation} 

 Let $\sd$ be the uniform measure (normalized surface measure) on $\Sbd$. Then  $\Phi(\sd)$, the pushforward measure of $\sd$ under $\Phi$, is the uniform measure on $\Dc$. In fact, $\Phi(\sd)$ is the Haar invariant measure induced by the unitary group (see \cite[Section 4.2]{CP15}).  

\section{Large $N$ behavior of optimal   configurations}\label{sec:cont}
We first focus on the continuous problem
\begin{equation}\label{equ:cont}
\min_{\mu\in\Mcc(\Sbd)}\iint_{\Sbd\times\Sbd}G_s(x,y)\ d\mu(x)d\mu(y),
\end{equation}
The results are of independent interest, and will be used in Section \ref{sec:goodframe}. 

For future reference, we set
$$I_s(\mu):=I_{R_s}(\mu)=\iint_{\Dc\times \Dc}R_s(x,y)\ d\mu(x)d\mu(y)$$  and
$$J_s(\mu):= I_{G_s}(\mu)=\iint_{\Sbd\times\Sbd}G_s(x,y)\ d\mu(x)d\mu(y).$$

As previously discussed,  the projective space $\Dc$  embedded in  $\R^m$ is a smooth ($C^\infty$) compact manifold, so Theorem \ref{lem:zero}  applies with $A=\Dc$ and $\alpha=\dim(\Dc)$.  

\begin{theorem}\label{thm:contP} For the projective Riesz kernel $G_s(x,y)$, the following properties hold.
\begin{enumerate}
\item[(1)] If $0\leq s<\dim (\Dc)$, then $\mu$ is a $G_s$-equilibrium measure on $\Sb^{d-1}$ if and only if its symmetrized measure $\mu_\sym$ is the normalized surface measure $\sd$ on $\Sbd$.
\item[(2)] If $s\geq \dim (\Dc)$, then $\Sb^{d-1}$ has $G_s$-capacity  0.
\item[(3)] Let  $s\ge 0$  and let $X_N^*$ be a $G_s$-optimal $N$-point configuration on $\Sbd$ for $N\ge 2$.  Then the sequence of normalized counting measures
$ \nu(\Phi(X_N^*))$ converges weak* to the uniform measure $\Phi(\sd)$ on $\Dc$ as $N\rightarrow\infty$.
\end{enumerate}
\end{theorem}
\begin{proof}
(1) When $s>0$, by \eqref{equ:red} and the definition of a pushforward measure, 
\begin{equation}\label{equ:UU}
U^{\mu}_{G_s}(x)=\int_{\Sbd}G_s(x,y)\ d\mu(y)=\int_\Dc\frac{2^{2/s}}{\|p_x-p\|^s}d\mu_\proj(p)=\text{const}\cdot U_{R_s}^{\mu_\proj}(p_x).
\end{equation}
A similar equality holds for the log case $s=0$: $U^{\mu}_{G_s}(x)=\text{const}*U_{R_s}^{\mu_\proj}(p_x)+\text{const}$. 

Thus for $s\ge 0$, the uniform measure $\sd$ produces a constant potential function with the kernel $G_s$, so $\Phi(\sd)$ also produces a constant potential function with the Riesz kernel $R_s$. By Theorem~\ref{lem:zero}(1), $\Phi(\sd)$ must be the unique minimizer of \eqref{equ:Rcont}. 

On the other hand, similar to \eqref{equ:UU},
\begin{align}\notag
J_s(\mu)&=\int_{\Sbd}\int_{\Sbd}G_s(x,y)\ d\mu(x)d\mu(y)\\
\label{equ:JI}&=\int_\Dc\int_\Dc\frac{2^{2/s}}{\|p -p'\|^s}d\mu_\proj(p)d\mu_\proj(p')=\text{const}\cdot I_s(\mu_\proj).
\end{align}
Again a similar equality holds for the log case. This implies that  $\mu$ is a minimizer of \eqref{equ:cont} if and only if $\Phi(\mu)=\mu_\proj$ is a minimizer of \eqref{equ:Rcont}, which has to be $\Phi(\sd)$.
By \eqref{equ:mu}, this is equivalent to $\mu_\sym=(\sigma_{d-1})_{\sym}=\sigma_{d-1}$.
This proves that $\mu$ is an equilibrium measure if and only if its symmetrized measure $\mu_\sym$ is $\sd$.

(2) With the relation \eqref{equ:JI}, this is a direct consequence of Theorem \ref{lem:zero}(2).

(3) For the discrete case, similar to \eqref{equ:JI}, we have $E_{G_s}(X_N)=\text{const}\cdot E_{R_s}(\Phi(X_N))+\text{const}$. So $X_N^*$ be a $G_s$-optimal $N$-point configuration on $\Sbd$ if and only if $\Phi(X^*_N)$ is an optimal configuration for the Riesz kernel $R_s$ on $\Dc$.

By Theorem \ref{thm:aym}, we conclude that the normalized counting measure $\nu(\Phi(X_N^*))$ converges to the $R_s$-equilibrium measure on $\Dc$ in the weak* sense. As shown in part (1), this unique equilibrium measure is $\Phi(\sd)$, when $s\in[0,\dim (\Dc))$. When $s\geq\dim (\Dc)$, by Theorem \ref{lem:zero}(2), we also have $\nu(\Phi(X_N^*))$ converges to $\Phi(\sd)$.


\end{proof}

\section{Discrete minimal energy problem}\label{sec:dis}
In this section we consider discrete extremal energy problems,  for a general class of projective kernels  of the form \eqref{equ:Kf}. Once again, the optimal configuration is an equivalent class in the sense of \eqref{equ:equiv}.  Theorem \ref{thm:circleP} is for the 1-dimensional sphere in the real vector space while Theorem \ref{thm:eat} is a general result over $\Hb$. Corollary \ref{cor:r2} addresses the special projective Riesz kernel case \eqref{equ:projective}.

%
%

\begin{theorem} \label{thm:circleP}
If $f:(0,\sqrt{2}]\rightarrow\R$ is a non-increasing convex function defined at zero by the (possibly infinite) value $\lim_{t\rightarrow0^+}f(t)$, then $X_N^{(h)}$ given in \eqref{equ:h} is an optimal configuration on $\R \Pb^1$ for the problem \eqref{equ:disg} where $K$ is as in \eqref{equ:Kf}. If, in addition, $f$ is strictly convex, then up to the equivalence relation in \eqref{equ:equiv}, no other $N$-point configuration is optimal.
\end{theorem}
\begin{proof}
By \eqref{equ:essential}, $K(x,y)=f(\|p_x-p_y\|)$, so we need to consider the minimal energy problem \eqref{equ:disg} with the kernel function to be $f(\|x-y\|)$ on the compact set $\Dc=\Phi(\Sbb^1)$.
The map $\Phi: \Sb^1\rightarrow \R\Mb_{2\times2}^h$ is precisely 
$$\Phi: (x,y)\rightarrow\left[\begin{array}{cc}x^2&xy\\xy& y^2
\end{array}\right].$$ 
As mentioned at the beginning of Section \ref{sec:embed}, $\R\Mb_{2\times2}^h$ is identified with $\R^3$ using the mapping $\left[\begin{array}{cc}x^2&xy\\xy& y^2
\end{array}\right]\rightarrow (x^2,\sqrt{2}xy,y^2)$. This way, $\Dc$ is a circle in $\R^3$  with radius $1/\sqrt{2}$.

With $r=1/\sqrt{2}$, the function $f$ satisfies the assumptions of Theorem \ref{thm:circleR}, so $$\sum_{i\neq j}K(x_i, x_j)=\sum_{i\neq j} f(\|p_{x_i}- p_{x_j}\|)$$ is minimized if 
 $p_{x_1}, p_{x_2},\ldots, p_{x_N}$ are equally spaced on the circle $\Dc$. One can  easily show that $\Phi$ maps equally spaced points on half $\Sb^1$ to equally spaced points on $\Dc$. So minimizers of \eqref{equ:disg} are precisely the equivalence class of equally spaced points on half of $\Sb^1$. 
 \end{proof}

\begin{remark}
It is well known that $\R\Pb^{d-1}$ is a compact Riemannian manifold. However,    $\R\Pb^{d-1}$  is topologically equivalent to a sphere only when $d=2$. \end{remark}

\begin{remark}
The frame potential kernel $|\langle x,y\rangle|^2$ can be written as $g(\sqrt{2-2|\langle x,y\rangle|^2})$, where $g(t)=1-t^2/2$ is not convex on $[0,\sqrt{2}]$. As a consequence, Theorem \ref{thm:circleP} cannot be applied to the frame potential kernel. The conclusion of Theorem \ref{thm:circleP} is however true, but there is no uniqueness (see \cite{BF03}). 
\end{remark}


The discrete minimal energy problem for the Riesz $s$-kernel is in general very hard as mentioned previously. The situation is slightly better for  kernels that are a function of absolute value of inner product, as we have the following general characterization when an \emph{equiangular tight frame} (ETF) exists. A frame $X=\{x_i\}_{i=1}^N$ is \emph{equiangular} if $\frac{|\langle x_i,x_j\rangle|}{\|x_i\|\|x_j\|}$ is a constant for all $i\neq j$. An ETF  is a frame that is equiangular and tight.  For frames in $\Sc(d, N)$, a necessary condition for the  existence of ETF is  $N\leq d(d+1)/2$ for $\Hb=\R$ and $N\leq d^2$ for $\Hb=\C$. The coherence has the famous Welch bound
\begin{equation}\label{equ:welch}
\xi(X)\geq\sqrt{\frac{N-d}{d(N-1)}}, \quad\text{for all }X\in\Sc(d,N),
\end{equation} 
and is achieved by ETFs. This can be easily derived from the relation (cf. \cite{CKM16})
\begin{equation}\label{equ:fpf}
N+N(N-1)\xi(X)^2\geq\sum_{i,j=1}^N|\langle x_i,x_j\rangle|^2=\|XX^*-\frac{N}{d}I_d\|_F^2+\frac{N^2}{d}\geq\frac{N^2}{d}.
\end{equation} 
The Welch bound also coincides with the simplex bound of the chordal distance in \cite{CHS96}.
We refer interested readers to \cite{STDH07} for more details and \cite{FM15} for a table on existing ETFs. The second inequality in \eqref{equ:fpf} also shows that the frame potential is minimized when the frame is tight.

The second theorem is for both the real and complex case. 
\begin{theorem}\label{thm:eat}
Let $\tilde f:(0,2]\rightarrow\R$ be a strictly convex and decreasing function defined at $t=0$ by the (possibly infinite) value $\lim_{t\rightarrow0^+}\tilde f(t)$, and $ f:(0,\sqrt{2}]\rightarrow\R$ be a strictly convex and decreasing function defined at $t=0$ by the (possibly infinite) value $\lim_{t\rightarrow0^+} f(t)$. If $N$ and $d$ are such that an ETF exists, then
\begin{enumerate}
\item[(i)]  it is the unique optimal configuration of \eqref{equ:disg} for the kernel $\widetilde{K}(x,y)=$ \\ $\tilde f(2-2|\langle x,y\rangle|^2)$;
\item[(ii)] it is also the unique optimal configuration of \eqref{equ:disg} for the kernel $K(x,y)=f(\sqrt{2-2|\langle x,y\rangle|^2})$. 
\end{enumerate}
\end{theorem}
\begin{proof}
(i) From \eqref{equ:essential}, $\widetilde{K}(x,y)=\tilde f(\|p_x-p_y\|^2)$.
Let $X=\{x_1, x_2,\ldots, x_N\}$ be an arbitrary configuration on the sphere and set $P_i:=p_{x_i}=x_ix_i^*$. Then, by \eqref{equ:inner},
\begin{align*}
J:&=\sum_{i\neq j}\|P_i-P_j\|^2=\sum_{i=1}^N\sum_{j\neq i}(2-2\langle P_i,P_j\rangle)= \sum_{i=1}^N \left(2(N-1)-2\sum_{j=1}^N \langle P_i,P_j\rangle+2\right)\\
&=2N^2-2\sum_{i,j=1}^N|\langle x_i,x_j\rangle|^2\leq 2N^2-2N^2/d,
\end{align*}
where the last inequality follows from \eqref{equ:fpf}. Thus, 
\begin{align}
&E_{\widetilde{K}}(X)=\frac{N(N-1)}{1}\cdot\frac{1}{N(N-1)}\sum_{i\neq j}\tilde f(\|P_i-P_j\|^2) \nonumber \\
&\geq \frac{N(N-1)}{1} \tilde f\left(\sum_{i\neq j}\frac{1}{N(N-1)}\|P_i-P_j\|^2\right)= N(N-1) \tilde f\left(\frac{1}{N(N-1)} J\right)\\
&\geq N(N-1)\tilde f\left(\frac{2N^2-2N^2/d}{N(N-1)}\right)=N(N-1)\tilde f\left(\frac{2N(1-1/d)}{(N-1)}\right).\nonumber
\end{align}

The first inequality becomes equality if and only if  $|\langle x_i,x_j\rangle|$ is   constant for   $i\neq j$; i.e., $X$ is equiangular. The second inequality becomes equality if and only if $X$ is a unit norm tight frame. Therefore, if an ETF exists for a given $d$ and $N$, then this ETF is the unique $\widetilde{K}$-energy minimizer.


Part (ii) is a direct consequence of (i). Indeed, $f(\sqrt{2-2|\langle x,y\rangle|^2})=g(2-2|\langle x,y\rangle|^2)$, where $g(t)=f(\sqrt{t})$.  Since $f$ is decreasing and convex, the same holds for $g$, to which we apply (i).
\end{proof}

\begin{remark}\label{rem:p}
Theorem \ref{thm:eat} shows that an ETF is {\em universally optimal} (see \cite{CK07,CKM16}) in the sense that it minimizes the energy for any potential that is a completely monotone function of distance squared in the projective space.  
  
The assumption of   Theorem \ref{thm:eat} (part (i)) is weaker than that of Theorem \ref{thm:circleP} as reflected in the above proof.  
For example, Theorem \ref{thm:eat} part (i) recovers Proposition 3.1 of \cite{EO12} since $|\langle x,y\rangle|^p=f(1-|\langle x,y\rangle|^2)$ with $f(t)=(1-t)^{p/2}$. It is easy to verify that $f(t)$ is decreasing and convex on $[0,1]$ when $p>2$. However $(1-t^2)^{p/2}$ is not convex, and therefore part (ii) cannot be used to recover Proposition 3.1 of \cite{EO12}. We refer the interested reader to \cite{CGGKO20} for more results on $p$-frame potential.
\end{remark}


Both Theorems \ref{thm:circleP} and \ref{thm:eat}  apply to the projective Riesz kernel since $\log {1}/{t}$ and $ {1}/{t^{s}}$ are strictly decreasing and strictly convex. 
\begin{corollary}\label{cor:r2}
For the projective Riesz $s$-kernel minimization problem \eqref{equ:projective} when $s\in[0,\infty)$,
\begin{enumerate}
\item[(i)]   the configuration $X^{(h)}_N$  defined in \eqref{equ:h} is   optimal for $\Sb^1\subset\R^2$;
\item[(ii)] if it exists, an ETF is the optimal configuration for $\Sbd\subset\Hb^d$.
\end{enumerate}
\end{corollary}

\begin{remark}
The conclusions of Corollary \ref{cor:r2} hold for $s=\infty$ (best line-packing problem). These results were mentioned in \cite{SH03} and are also implied by Theorem \ref{thm:sinfP}.
\end{remark}


In particular, the optimal configuration of $N=d+1$ points that solves \eqref{equ:projective} is given by the vertices of a regular $d$-simplex because it is an ETF.  When $d=3$, the results for ETF are well known for small values of $N$. We summarize these results in Table \ref{tab:pr}, where we also compare the optimal configurations of the projective Riesz kernel and the classical Riesz kernel.
They only share the same optimal configuration for the $N=d+1$ case. Moreover, ETFs are optimal configurations for the projective kernel while nothing is known for the Riesz kernel in general. 

We have explained intuitively why the projective Riesz kernels are better at promoting well-separated frames than the Riesz kernel. This is reflected in Table \ref{tab:pr}. For the first case $\Sb^1$ when $N=4$, the optimal configuration for the projective Riesz kernel is two orthonormal bases with a 45 degree angle, which is a well-separated tight frame, while the optimal configuration for the Riesz kernel consists of 4 equally spaced points on $\Sb^1$. For the third row, an orthonormal basis is the optimal frame whereas 3 points on one great circle is not even a frame.  For 6 points on a sphere, the Riesz optimal configuration is again two copies of the same orthonormal bases. More numerical support can be found in Section \ref{sec:num}.

\begin{table}[htb]
\caption{Optimal configuration comparison on $\Sbd\subset\R^d$} \label{tab:pr}
\centering 
\small
\begin{tabular}{l|c|c|c}
           & proj. Riesz kernel  & Best line-packing & Riesz kernel \\
           &$s\in[0,\infty)$ &$s=\infty$ & $s\in[0,\infty)$\\
           \hline
           \hline
$\Sb^1, \text{any } N$ & \multicolumn{2}{c|}{equally spaced points on half circle} & equally spaced points on $\Sb^1$\\
\hline
$\Sb^2, N=2$& \multicolumn{2}{c|}{two orthogonal points} & two antipodal points\\
\hline
\multirow{2}{*}{$\Sb^2, N=3$}& \multicolumn{2}{c|}{\multirow{2}{*}{any orthonormal basis}} & vertices of an equilateral  \\ 
            & \multicolumn{2}{c|}{} & triangle on a great circle\\
            \hline
$\Sb^2, N=4$& \multicolumn{3}{c}{vertices of a regular tetrahedron (simplex)} \\
\hline
\multirow{2}{*}{$\Sb^2, N=5$}& \multirow{2}{*}{open, see Table \ref{tab:5}} &removing any vector from& \multirow{2}{*}{partially solved in \cite{S16}}\\
&&the $3\times6$ ETF, see \cite{CHS96, BD06} &\\
\hline
$\Sb^2, N=6$& \multicolumn{2}{c|}{$3\times6$ ETF, or vertices of the icosahedron}  & octahedral vertices\\
\hline
$\Sb^{d-1},N=d+1$& \multicolumn{3}{c}{vertices of the simplex}\\
\hline
$\Sb^{d-1},N$ & \multicolumn{2}{c|}{$d\times N$ ETF when exists} & open
\end{tabular}
\end{table}

The Grassmannian frame consisting of 5 vectors is constructed by removing an
arbitrary element of the optimal Grassmannian frame consisting of 6 vectors (ETF). The coherence of the Grassmannian frame (in both $N=5$ and $N=6$) is $1/\sqrt{5}$.   We refer
to \cite{CHS96} for more details.

It is not possible for 5 points to be an ETF in $\R^3$, and the optimal configuration of \eqref{equ:projective} remains open to the best knowledge of the authors. The numerical experiments in Table~\ref{tab:5} indicate that optimal configurations have exactly  2 distinct inner products. These   inner products  depend  on the value $s$.  As $s\to \infty$, Theorem \ref{thm:sinfP} below implies that the inner products converge to $1/\sqrt{5}$.  

\begin{table}[htb]
\caption{Optimal configurations for projective Riesz kernel and best line packing when $N=5, d=3$.}\label{tab:5}
\begin{tabular}{c|ccccc}
&$s=2$ & $s=10$ & $s=15$ & $s=\infty$\\
\hline
$\{|\langle x_i,x_j\rangle|: i\neq j\}$& $\{0.293, 0.506\}$&$\{0.366, 0.478\}$&$\{0.389, 0.471\}$&$\{1/\sqrt{5}\approx0.447\}$\\
\end{tabular}

\end{table}

We conjecture that if $X^*=\{x_1,x_2,x_3,x_4,x_5\}$ is an optimal configuration of \eqref{equ:projective} for $s\in[0,\infty)$ and $N=5$, then the cardinality of the set $\{|\langle x_i,x_j\rangle|,i\neq j\}$ is 2. We remark that constructions of biangular tight frames are studied in \cite{CCHT17}.

\section{Optimal configurations as frames}\label{sec:goodframe}

We show in this section that frames rising from \eqref{equ:projective} are  well-separated  and nearly tight asymptotically. 
Since the frame vectors will always be on the sphere, it is understood that $\Ec_{G_s}(N)$ refers to $\Ec_{G_s}(\Sbd, N)$. 


The following theorem is the analog of Theorem \ref{thm:sinf} for the projective Riesz $s$-kernel. It can be over the real or complex field.
\begin{theorem}\label{thm:sinfP}
The best line-packing problem is the limit of problem \eqref{equ:projective}   as $s\to \infty$:
$$\lim_{s\rightarrow\infty}\Ec_{G_s}(N)^{1/s}=\sqrt{\frac{1}{1-\xi_N}}.$$
If, for $s>0$,   $X_s$ is an optimal configuration achieving $\Ec_{G_s}(N)$, then every cluster point as $s\rightarrow\infty$ of the set $\{X_s\}_{s>0}$ is a Grassmannian frame.
\end{theorem}
\begin{proof}
By \eqref{equ:red}
$$\Ec_{G_s}(N)=2^{s/2}\Ec_{R_s}(\Dc,N)$$
Taking the $s$th root and letting $s\rightarrow\infty$, we have 
$$\lim_{s\rightarrow\infty}\Ec_{G_s}(N)^{1/s}=\lim_{s\rightarrow\infty}\sqrt{2}\Ec_{R_s}(\Dc,N)^{1/s}=\frac{\sqrt{2}}{\delta_N(\Dc)}=\sqrt{\frac{1}{1-\xi_N}}.$$
The last two equalities are from Theorem \ref{thm:sinf} and \eqref{equ:2},

The second assertion is also a consequence of Theorem \ref{thm:sinf} since  for any $G_s$-optimal configuration  $\{x_i\}\subset \Sbd$, the configuration $\{p_i=x_ix_i^*\}$  is  $R_s$-optimal   for $ \Dc$.
\end{proof}

%

The Grassmannian frames have the best separation by definition, but Theorem \ref{thm:sinfP} suggests that we are also able to find well-separated frames by solving \eqref{equ:projective} for large values of $s$. 
We will further show that projective Riesz energy minimizing frames are well-separated in the sense that their coherence have  optimal   order asymptotic growth (Theorem \ref{thm:sepp}).


Let $B(x,r)\subset\R^m$ be the   ball centered at $x$ with radius $r$. For a number $\alpha>0$ and a positive Borel measure $\mu$ supported on $A\subset\R^m$, we say that $\mu$ is {\em upper $\alpha$-regular} if there is some finite constant $C_A$ such that
\begin{equation}\label{equ:D}
\mu(B(x,r))\leq C_Ar^{\alpha}\qquad \text{ for all } x\in A, \, 0<r\leq \text{diam}(A),
\end{equation}
and similarly that $\mu$ is {\em lower  $\alpha$-regular} if there is some positive constant $c_A$ such that
\begin{equation}\label{equ:c}
\mu(B(x,r))\geq c_Ar^{\alpha}\qquad \text{ for all } x\in A,\, 0<r\leq \text{diam}(A).
\end{equation}
 It is not difficult to verify that  $\Phi(\sd)$, the uniform measure on $\Dc$, is both upper and lower $(\dim\Dc)$-regular (see the Appendix).  

We recall that $\delta_N(\Dc)$ is the maximum of the separation distance among all possible $N$ point configurations on $\Dc$. Since $\Phi(\sd)$ is lower $(\dim \Dc)$-regular, there is a constant $C<\infty$ such that
$\delta_N(\Dc)\leq CN^{-\frac{1}{\dim \Dc}}$ for all $N$,  
which follows immediately   by observing that  the $\Phi(\sd)$ measure of an arbitrary packing   in $\Dc$   is no more than $\Phi(\sd)(\Dc)=1$.
 Expressing this bound in terms of 
coherence  (recall \eqref{equ:2}) gives
\begin{equation}\label{equ:optimal}
\xi_N\geq 1-\frac{C^2}{2}N^{-\frac{2}{\dim \Dc}}.
\end{equation}
The above bounds are   attained by any sequence of best-packing configurations on $\Dc$ (e.g., see \cite[Chapter~13]{BHS20}).

Based on the above observation, we say that  a sequence $(X_N)$ of $N$-point configurations in $\Dc$ is   {\em well-separated} if there is some constant $\widetilde C>0$ such that $\delta(X_N)\ge \widetilde C N^{-\frac{1}{\dim \Dc}}$ for all $N$.
  Equivalently, in terms of coherence, $\left(X_N \right)$ is  well-separated if   
\begin{equation}
\xi(X_N)\leq 1-\frac{\widetilde C}{2}N^{-\frac{2}{\dim \Dc}},
\end{equation}
for all $N$.

We will show that the projective $s$-Riesz energy minimizing points are  well-separated  when $s>\dim \Dc$. This is a consequence of the 
following known theorem  for  optimal configurations  on  more general sets. 

\begin{theorem}[{\cite[Corollary 2]{HSW12}}] \label{thm:sepr} Suppose $A\subset\R^m$ is compact  and supports an upper $\alpha$-regular   measure $\mu$ as in \eqref{equ:D}. Let $s>\alpha, N\geq 2$ be fixed. If $X_N^*$ is an $N$-point minimizing configuration on $A$ for the $s$-Riesz energy minimizing problem \eqref{equ:riesz}, then
\begin{equation}
\delta(X_N^*)\geq C_1 N^{-\frac{1}{\alpha}},
\end{equation}
where $C_1=\left(\frac{\mu(A)}{C_A}(1-\frac{\alpha}{s})\right)^{1/\alpha}\left(\frac{\alpha}{s}\right)^{\frac{1}{s}}$.
\end{theorem}

We will next apply Theorem~\ref{thm:sepr} to obtain the following bound on the coherence of optimal $G_s$ configurations.  
\begin{theorem}[Separation]\label{thm:sepp}
Let $s>\dim \Dc$.  If  $X_s$ is an $N$-point minimizing configuration of \eqref{equ:projective}, then
\begin{equation}\label{equ:sepp}
\xi(X_s)\leq 1-\frac{C_2^2}{2}N^{-{2}/{\dim\Dc}},
\end{equation}
where the constant $C_2$ is independent of $N$  and, in the case $\Hb=\R$, can be found in \eqref{equ:C2}.  Consequently, any sequence of such configurations is 
well-separated as $N\to \infty$.
\end{theorem}
\begin{proof}
By \eqref{equ:red}, if $X_s=\{x_i\}_{i=1}^N$ is an optimal configuration of \eqref{equ:projective}, then $P_s=\{x_ix_i^T\}_{i=1}^N$ is an optimal configuration of \eqref{equ:rieszC}. Appealing to Theorem \ref{thm:sepr} with $A=\Dc$, the projective space embedded in $\Hb\Mb_{d\times d}^h$, and recalling that $\Phi(\sd)$ is  upper   $(\dim\Dc)$-regular with constant $C_{\Dc}>0$, we have 
\begin{equation}\label{equ:thmsep}
\delta(P_s)\geq C_2 N^{-\frac{1}{\dim \Dc}},
\end{equation}
where 
$C_2:=\left(\frac{1}{C_\Dc}(1-\frac{\dim \Dc}{s})\right)^{\frac{1}{\dim \Dc}}\left(\frac{\dim \Dc}{s}\right)^{\frac{1}{s}}$.
The inequality \eqref{equ:sepp} then follows from \eqref{equ:coh}.  

In the case $\Hb=\R$, as shown in the Appendix  (see \eqref{equ:mupAp}), $$C_\Dc=\frac{2}{d-1}\gamma_d=\frac{2\Gamma(\frac{d}{2})}{(d-1)\Gamma(\frac{d-2}{2})\Gamma(1/2)},$$  where $\Gamma(\cdot)$ is the Gamma function. Therefore
\begin{equation}\label{equ:C2}
C_2=\frac{(d-1)(s-d+1)\Gamma(\frac{d-2}{2})\Gamma(1/2)}{2s\Gamma(\frac{d}{2})}\left(\frac{d-1}{s}\right)^{\frac{1}{s}}.
\end{equation}
\end{proof}

It is worth  noting in the case $\Hb=\R$ that the expected coherence of    an i.i.d. random   frame $X\in \Sc(d,N)$  generated
from the uniform distribution on the sphere  satisfies $\E[\xi(X)]\approx 1-C_dN^{-\frac{4}{d-1}}$ (see Appendix), which is significantly worse than optimal. This fact is also demonstrated numerically in  Section \ref{sec:num} (see Figure \ref{fig:640}).
%

\color{black}
We next show that the optimal configurations of \eqref{equ:projective} are nearly tight. 

\begin{theorem}[Nearly tight]\label{thm:tight}
Let $s\ge 0$. If   $X_N=\left \{x_1,\ldots,x_N\right\}\in \Sc(d,N)$ is  any optimal configuration for \eqref{equ:projective} for $N\ge 2$, then (treated as a matrix)
\begin{equation}\label{equ:asytight}
\lim_{N\rightarrow\infty}\frac{1}{N}X_NX_N^*=\frac{1}{d}I_d.\end{equation}
\end{theorem}

\begin{proof}
  Theorem \ref{thm:contP}(3) states that $\nu(\Phi(X_N))=\frac{1}{N}\sum_{i=1}^N\delta_{\Phi(x_i)}$  converges weak* to $\Phi(\sd)$. Thus for every continuous function $f$ defined on $\Dc$, 
$$\lim_{N\rightarrow\infty}\int_\Dc fd\nu(\Phi(X_N))=\int_\Dc fd\Phi(\sd).$$
 By the definition of a pushforward measure, this can be simplified to
 $$\lim_{N\rightarrow\infty}\int_{\Sb^{d-1}} f\circ\Phi \ d\nu(X_N)=\int_{\Sb^{d-1}} f\circ\Phi \ d\sd;$$
 that is,
\begin{equation}\label{equ}
\lim_{N\rightarrow\infty}\frac{1}{N}\sum_{i=1}^Nf(\Phi(x_i))=\int_{\Sb^{d-1}} f(\Phi(x))d\sd.
\end{equation}
Let $f(\Phi(x))=xx^*$ be a vector-valued function. Then \eqref{equ} implies
\begin{equation}\label{equ:xxt}
\lim_{N\rightarrow\infty}\frac{1}{N}\sum_{i=1}^Nx_ix_i^*=\int_{\Sb^{d-1}} xx^*d\sd.
\end{equation}
We need to integrate every entry of the right-hand side. Let $x=(x(1), x(2), \ldots, x(d))^T\in\Sbd$, so that $xx^*=(x(i)\overline{x(j)})$.  Then,
\begin{align*}
\text{if }i= j,& \quad\int |x(j)|^2d\sd=\frac{1}{d}\int(|x(1)|^2+\cdots+|x(d)|^2)d\sd=\frac{1}{d};\\
\text{if }i\neq j, &\quad\int x(i)\overline{x(j)}d\sd=0 \text{ by symmetry.}
\end{align*}
 Since $X_NX_N^*=\frac{1}{N}\sum_{i=1}^Nx_ix_i^*$, from \eqref{equ:xxt} we deduce that

$$\lim_{N\rightarrow\infty}\frac{1}{N}X_NX_N^*=\frac{1}{d}I_d.$$
\end{proof}

\color{black}

Theorem \ref{thm:tight} says that the optimal configurations are nearly tight asymptotically as $N\rightarrow\infty$ in relation to \eqref{equ:tight}. More desirable would be a stronger result of the form
\begin{equation}\label{equ:asytight2}
\left\|X_NX_N^*-\frac{N}{d}I_d\right\|_F=\mathcal{O}(N^{-q}), \quad q>0.
\end{equation}
 The numerical experiments in Section \ref{sec:num} (left side of Figure \ref{fig:tight}) do indeed suggest a result like \eqref{equ:asytight2} holds at least for small values of $s$. 
The work \cite{BSSW} provides a partial explanation for this phenomenon. It studies the convergence rate of \eqref{equ} for  $f$ in a Sobolev space (which is the case for every entry of $xx^*$). The numerical experiments therein suggests that the $s$-Riesz minimizing configurations (when $s=0, 1$) achieve the optimal order quasi Monte Carlo error bounds.

Regarding random tight frames, it is shown in \cite[Corollary 3.21]{E12} that $$\E\left(\left\|X_NX_N^*-\frac{N}{d}I_d\right\|_F^2\right)=N(1-\frac{1}{d}),$$ which grows as $N$ grows. Section \ref{sec:num} (right of Figure \ref{fig:tight}) shows that optimal configurations of \eqref{equ:projective} also outperforms random configurations on tightness.

\color{black}
\section{numerical experiments}\label{sec:num}
The numerical experiments conducted consider points in the real vector space, and were executed in Matlab. When solving  \eqref{equ:projective} (or \eqref{equ:projective:negative} for negative $s$), spherical coordinates are used so that the command \verb"fminunc" (unconstrained minimization) can be employed. Four experiments were performed. 

The first and the second experiments deal with the separation and tightness of the optimal configurations of \eqref{equ:projective} or \eqref{equ:projective:negative}, and are explained in Sections \ref{sec:num:sep} and \ref{sec:num:tight}. 
Since the objective function has lots of local minima, in both experiments, we run \verb"fminunc" with multiple random initializations to obtain a putative minimum. We then test whether the optimal configurations are nearly tight or have small coherence (well-separatedness). 

The third experiment presents an algorithm for obtaining tight frames with good separation. As explained in Section \ref{sec:num:both}, it is crucial to use a well-separated frame as an initialization.

\subsection{Good separation}\label{sec:num:sep}
The first experiment explores the asymptotic behavior of the coherence $\xi(X)=\max_{i\neq j}|\langle x_i, x_j\rangle|$ of projective Riesz minimizing points for various values of $s$ as $N$ gets larger. The result  displayed in Figure~\ref{fig:sep} is for $d=3$ with points   on $\Sb^2$. The number of points $N$  ranges from 3 to 100. The separation result Theorem \ref{thm:sepp} only applies to $s>2$, but our numerical experiment shows that the log case and $s=1$ case are achieving smaller coherence. The $s=-2$ (frame potential) case has the worst behavior as its minimizers could contain repeated (or antipodal) points. Notice that the coherence gets smaller as $s$ increases which is consistent with Theorem \ref{thm:sinfP}. Finally, the coherence curve was fit with $y=1- {3}/{N}$, which reflects Theorem \ref{thm:sepp}.
\begin{figure}[htb]
\includegraphics[width=.9\textwidth]{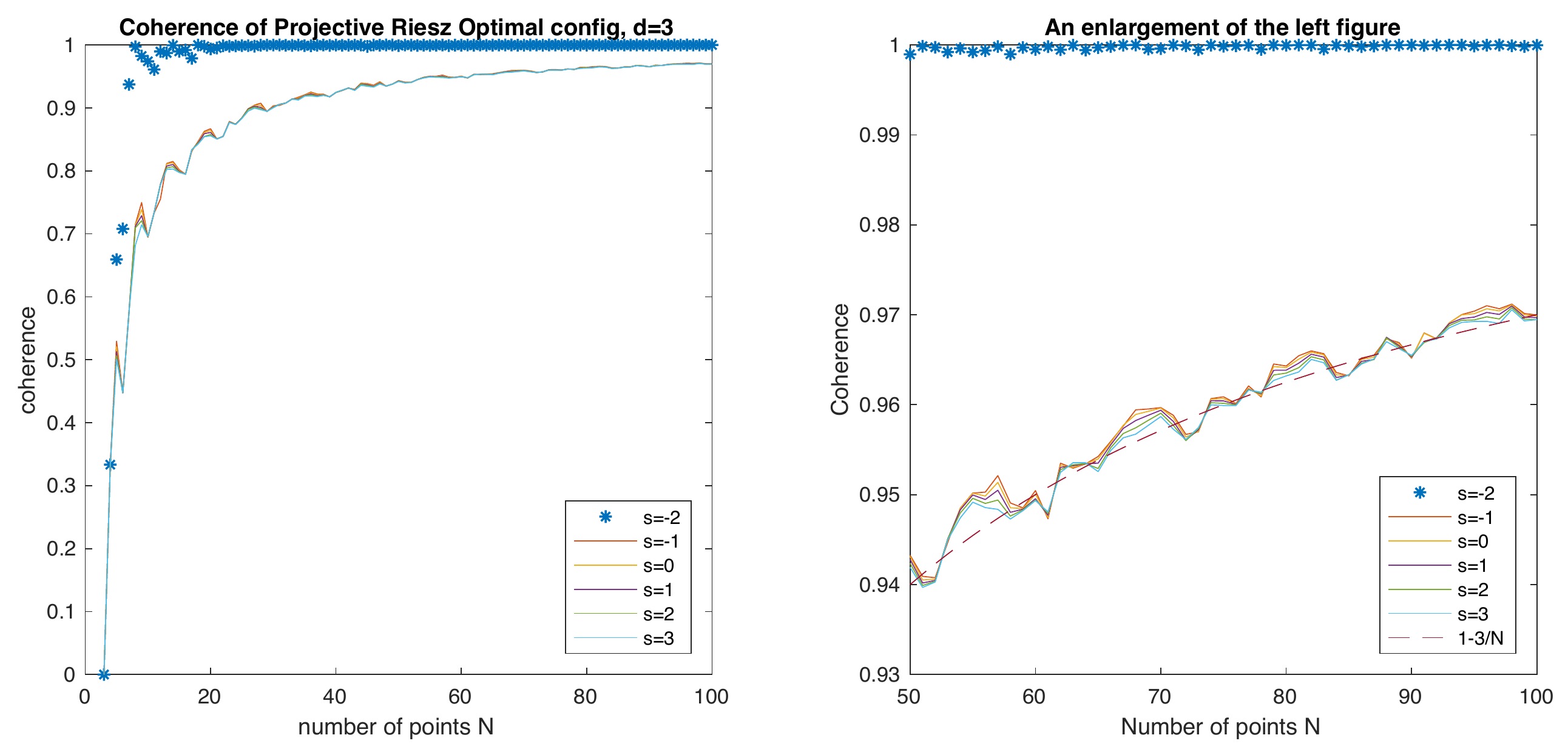}
\caption{}\label{fig:sep}
\end{figure}

The second experiment computed the coherence of the projective Riesz minimizing points for $d=6$ and relatively small values of $N$ (from 6 to 40), as shown in Figure \ref{fig:640}. Various $s$ are computed and compared with the Welch bound \eqref{equ:welch}, 
the Levenstein bound~\cite{L92,ZDL14}
$$\xi(X)\geq\sqrt{\frac{3N-d^2-2d}{(d+2)(N-d)}},\quad\text{if }N>d(d+1)/2,$$
and the Sloane database http://neilsloane.com/grass/. The Sloane database has the best known line-packings or the smallest coherence given $d, N$, among which some are only putatively known. Figure \ref{fig:640} also includes uniform random configurations. For each $N$, we display the coherence that is  averaged over 20 samples. 
We again observe that larger $s$ produces better separated frames, and $s=-2$ (frame potential case)  produces highly correlated frames.
For all values of $s$ except for -2, \eqref{equ:projective} achieves the Welch bound when $N=6, 7,16$ (these are all the ETFs and thus universally optimal), and it achieves the Levenstein bound when $N=36$.
The 36 point configuration in $\R^6$ is the 6-dimensional lattice $E_6$~\cite{FJM17} and is also known to be universally optimal~\cite{CKM16} although not an ETF.  We further remark that our numerical experiments  suggest that the Sloane grassmannian  configurations for $d=6$ and $N=12$ and $N=22$ may be universally optimal.  This might also be anticipated from    Figures~\ref{fig:640} and \ref{fig:tight6}.  These figures might also suggest the universal optimality of the Sloane  configuration for $N=21$, however this turns out not to be the case.
\begin{figure}[htb]
\includegraphics[width=.9\textwidth]{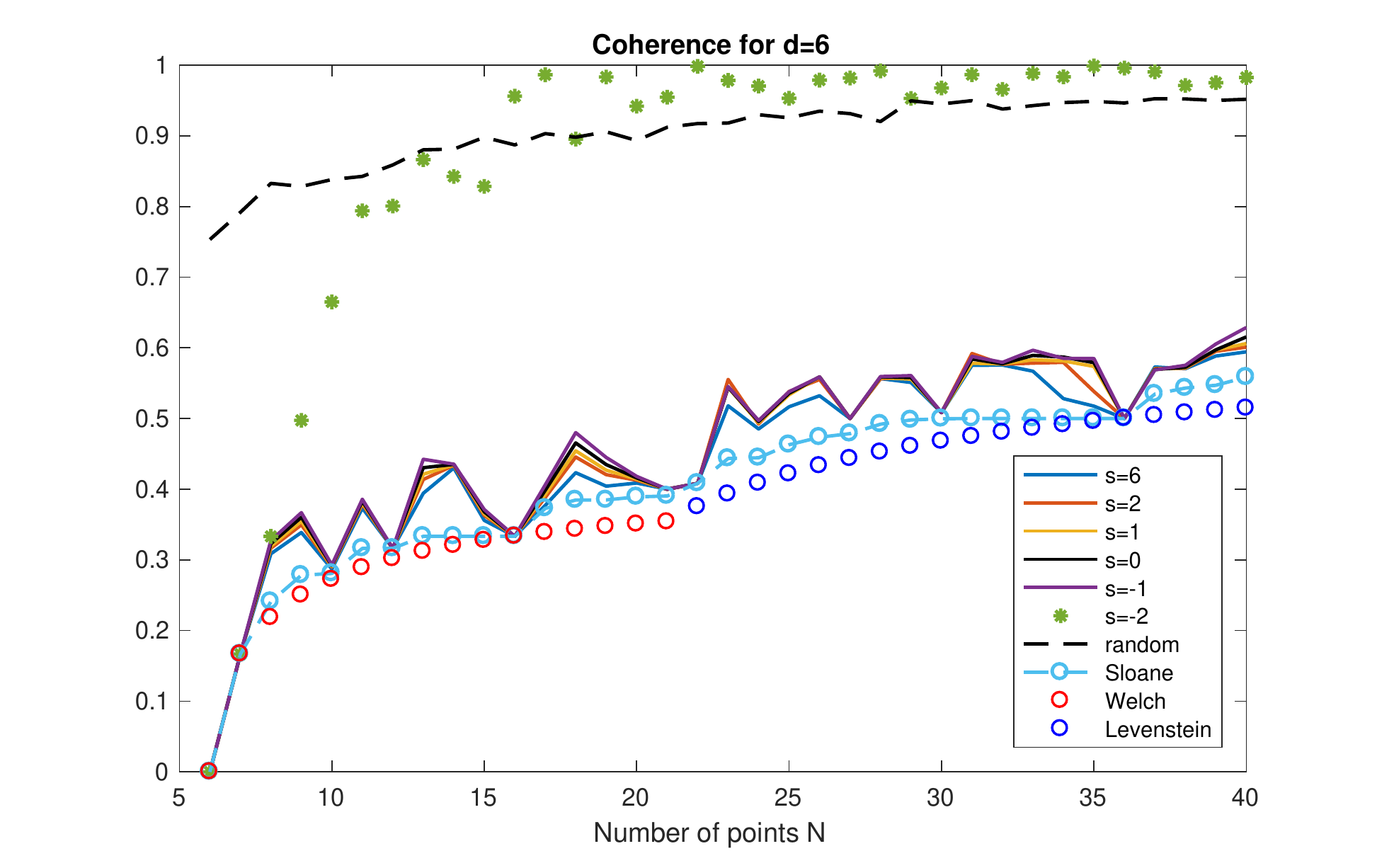}
\caption{}\label{fig:640}
\end{figure}

\begin{figure}[htb]
\includegraphics[width=1.\textwidth]{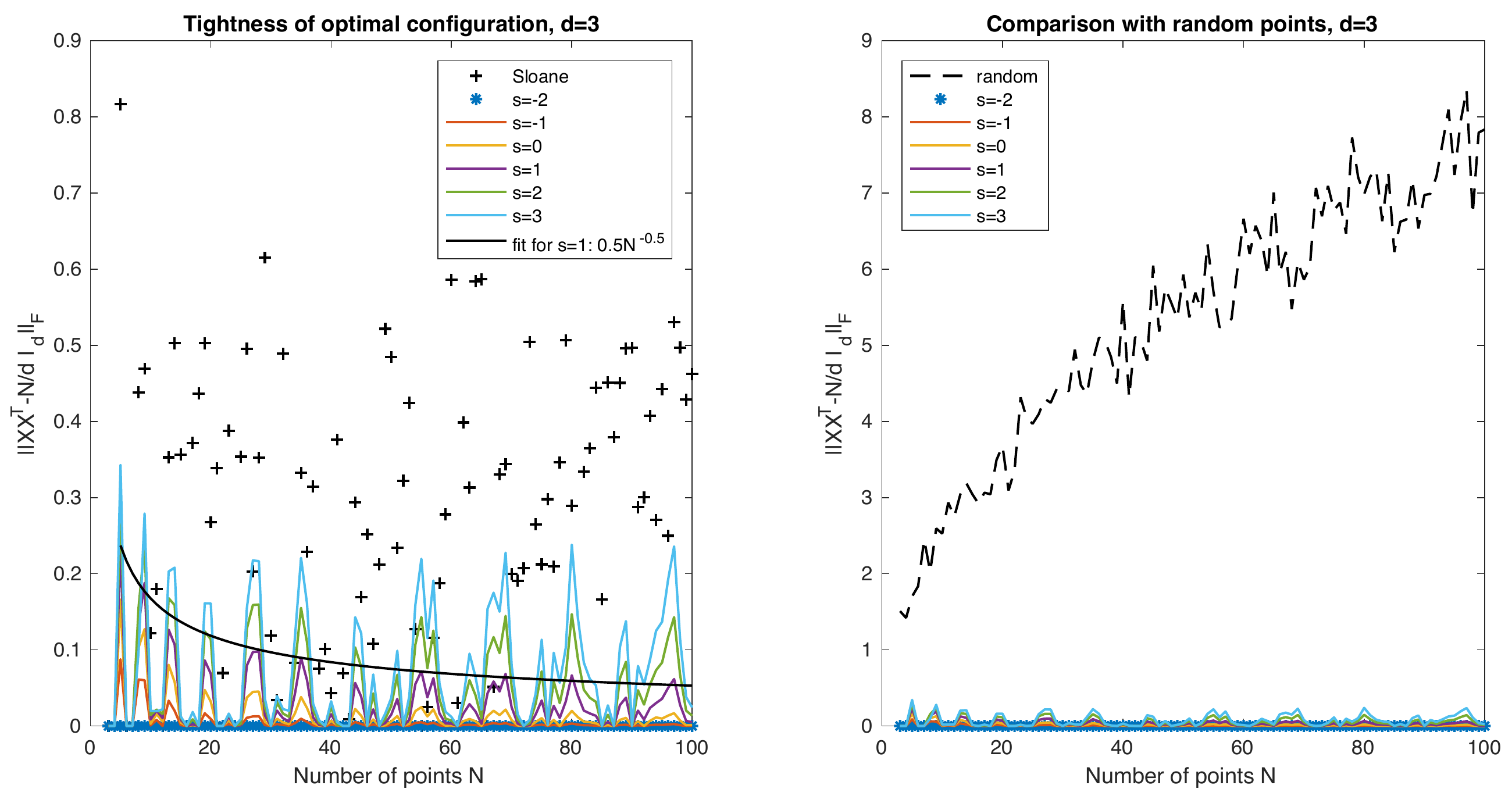}
\caption{Tightness for $d=3$. The right side plot includes random configurations.}\label{fig:tight}
\end{figure}

\begin{figure}[htb]
\includegraphics[width=0.8\textwidth]{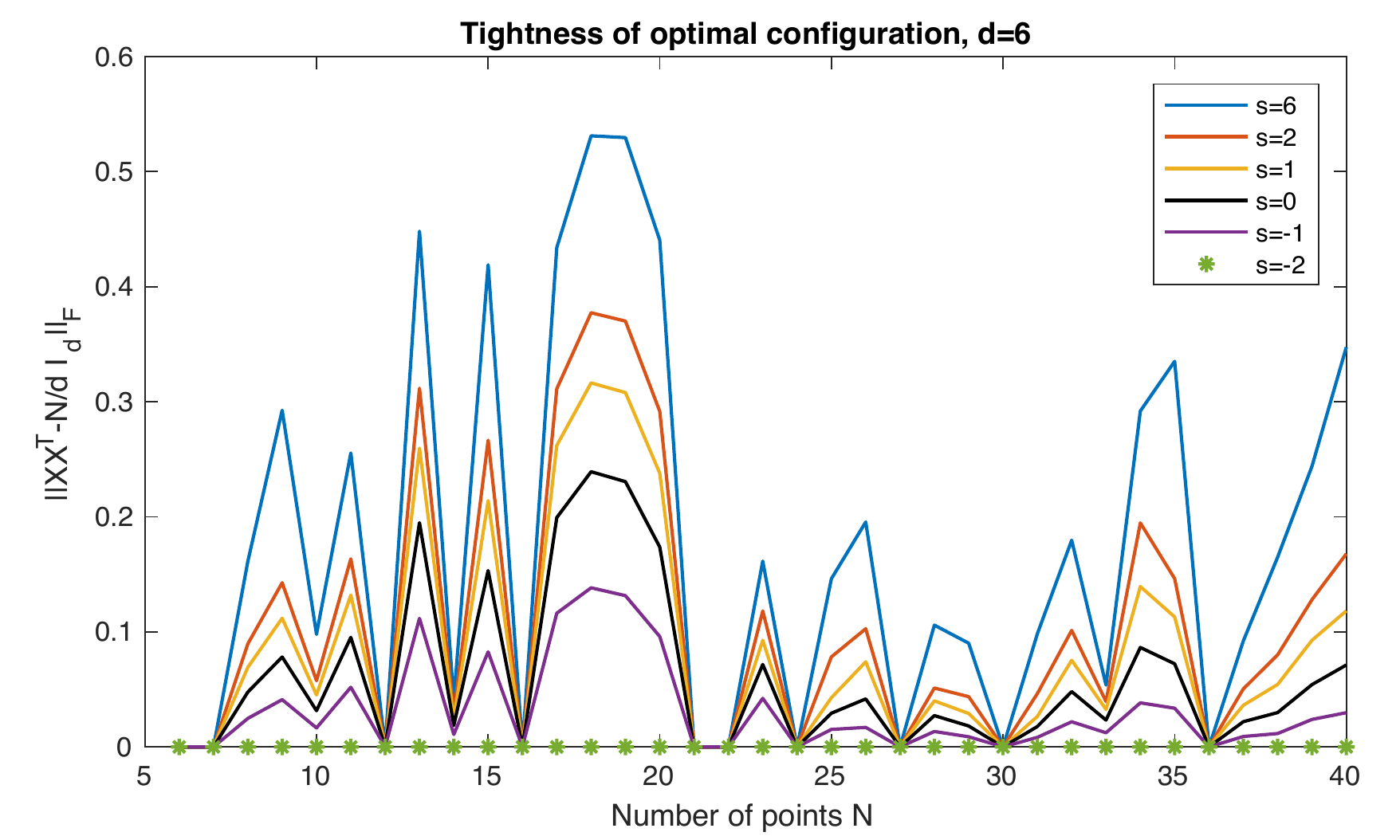}
\caption{Tightness for $d=6$. }\label{fig:tight6}
\end{figure}

\subsection{Nearly tight}\label{sec:num:tight}
We addressed the tightness of the optimal configurations   by computing $\|XX^T-\frac{N}{d}I_d\|_F$. We reuse the points generated from the first experiment with $d=3$ and $N$ ranging from 3 to 100.
 Figure \ref{fig:tight} (left) shows the results for $s=-2, -1,0, 1, 2, 3$. The $s=-2$ (frame potential) case recovers tight frames since by   \eqref{equ:fpf}, Riesz $(-2)$-energy is equal to  $\|XX^T-\frac{N}{d}I_d\|_F^2$ plus a constant. Unfortunately, the separation property deteriorates as  $s$ decreases while the  tightness property improves. This is further validated by the poor  tightness of the Sloane points as they correspond to the $s=\infty$ case.
 A least squares curve fitting was also performed for the peaks (least tight) for $s=1$, which exhibits an $N^{-1/2}$ decay, a better rate than what Theorem \ref{thm:tight} guarantees.
Notice that the right side  of Figure \ref{fig:tight} also includes    uniform random vectors for comparison.  The randomly generated configurations   exhibit worse behavior for both coherence and tightness.  This has  also been observed in \cite{BSSW}.

The tightness for $d=6$ with $N$ ranging from 6 to 40 is illustrated in Figure \ref{fig:tight6} where the points generated from the second experiment are reused.  The figure displays a clear pattern of improved tightness as $s$ decreases.


\subsection{Achieving good separation and exact tightness}\label{sec:num:both}

In this section we present experiments based on a simple algorithm for obtaining frames with good separation and exact tightness.   For the frame potential minimization problem ($s=-2$ of \eqref{equ:projective:negative}), recall that every local minimizer is a global minimizer; i.e.,  a  tight frame~\cite{BF03}.
The output of \verb"fminunc" is certainly affected by the initial input. As seen in Figure \ref{fig:sep} and Figure \ref{fig:640}, the tight frames found by minimizing the frame potential using random initializations generically have poor separation. We propose the following approach for generating well-separated tight frames in 
$\mathcal{S}(d,N)$ for given $d$ and $N$. 

\begin{enumerate} 
\item Generate a random frame $X\in \mathcal{S}(d,N)$. 
\item   Using $X$ as an initial configuration, use an optimization algorithm (such as gradient descent) to find a local minimizer $Y$ for \eqref{equ:projective} for some $s>d-1$.  Motivated by Theorems~\ref{thm:sepp} and \ref{thm:tight} the minimizer $Y$  is expected to be well-separated  and nearly tight. 
\item Minimize \eqref{equ:projective:negative} with $s=-2$ using $Y$ as the initial configuration.  The experiments presented  below suggest that the resulting frame is well-separated and tight.\\
\end{enumerate}

Variations on this approach such as iterating steps 2 and 3 or minimizing Riesz-$s$ energy  restricted to the manifold of tight frames will be explored in future work.

\begin{figure}
\centering
\includegraphics[width=1\textwidth]{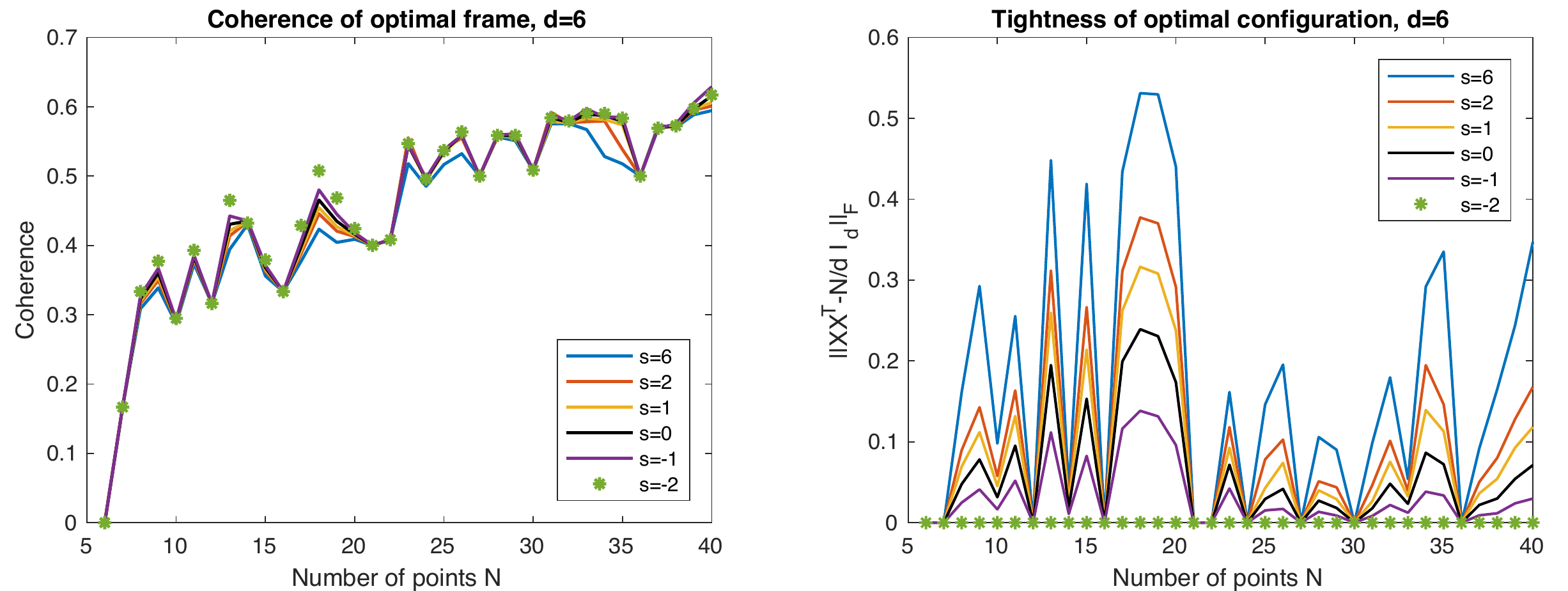}
\caption{The optimal configuration for $s=-2$ with for various choice of   $s$ used in step 2 for initialization.}\label{fig:sn2better}
\end{figure}

 Figure \ref{fig:sn2better} shows the performance of the optimal configuration of minimizing frame potential with the initialization being the optimal configuration obtained through solving \eqref{equ:projective}. The numerics indicate that these optimal tight configurations are indeed well-separated.
  The left graph of Figure \ref{fig:sn2better} should be compared to Figure \ref{fig:640} (the values of $s$ used in step 2 of the above algorithm are indicated in the figure and include values of $s<d-1$). Numerically, this is a promising way to find well-separated tight frames, which has many applications including signal transmission \cite{GKK01, HP04}. 

\begin{figure}[bht]
\centering
\includegraphics[width=0.8\textwidth]{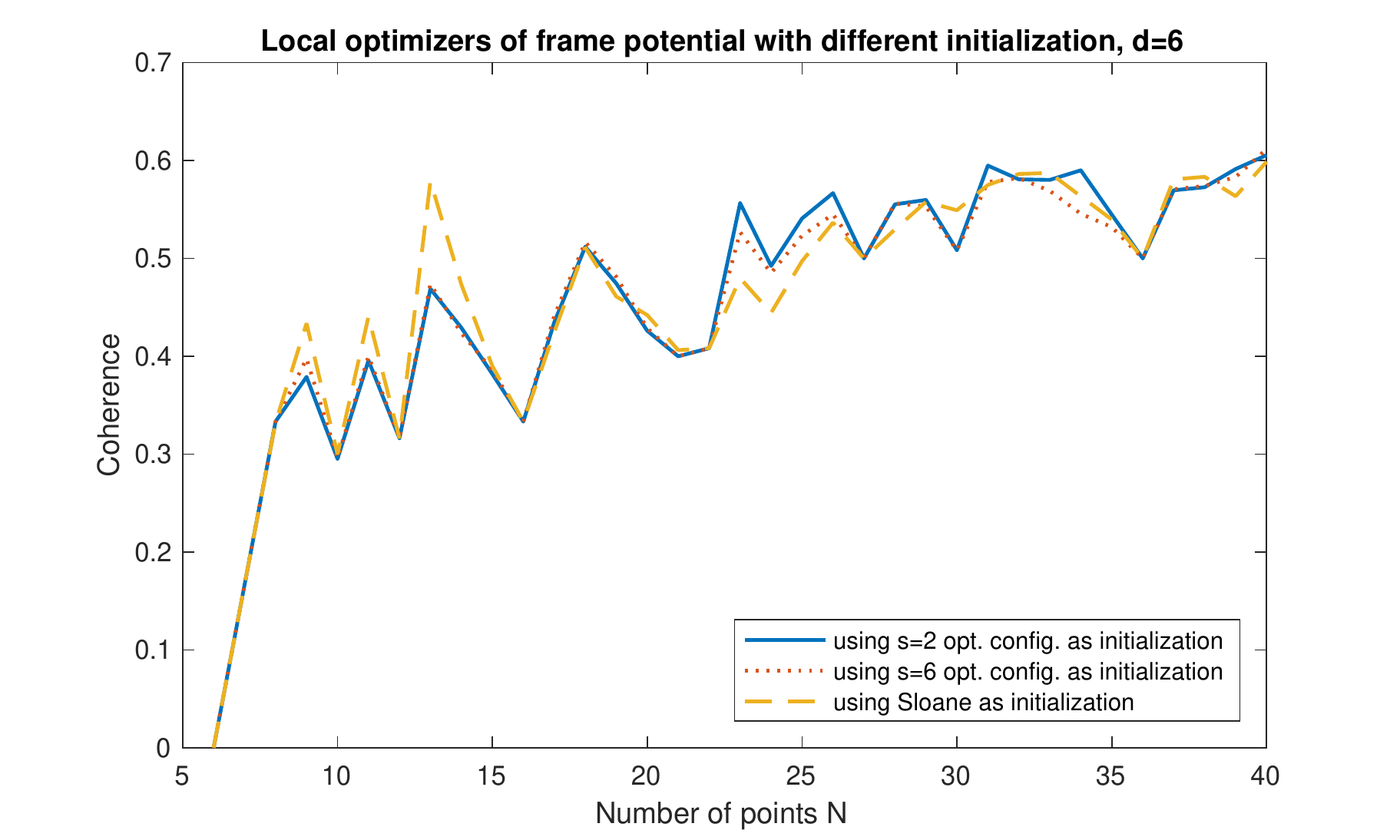}
\caption{Separation comparison of minimizing frame potential with different initializations}\label{fig:initialization}
\end{figure}

The set of finite unit norm frames is topologicaly connected, and an irreducible variety~\cite{CMS17}, but our experiment suggests that for a local minima of \eqref{equ:projective}, there will be a tight frame close to it. This can perhaps be explained by the recently solved Paulsen Problem~\cite{HM18}, which implies that for a nearly  tight unit norm frame $F$, there exists a unit norm tight frame nearby since we have shown that a local minima of \eqref{equ:projective} is nearly tight.
This suggests that step 3 will result in a nearby tight configuration if the configuration from step 2 is nearly tight as indicated in Theorem~\ref{thm:tight} for $N$ large. 


To further see how the initializations impact the frame potential problem, Figure \ref{fig:initialization} compares the coherence of the optimal configuration of solving \eqref{equ:projective:negative} ($s=-2$) with different initializations. The results are similar, but note that starting with Sloane points, the best separated points among the three, does not necessarily end up with the best separation.

\appendix
\section{}
\subsection{Uniform measure}
Given the hypersphere $\Sb^{d-1}$, let $C_r(x)$ be the hyperspherical cap centered at $x$, with $r$ being the Euclidean distance of the furthest point to $x$. That is,
$$C_r(x)=\{y\in\Sbd: \|x-y\|\leq r\}.$$
Recalling that $\sd$ denotes the normalized surface measure,  the following asymptotic formula holds:
\begin{equation}\label{equ:cap}
\sd(C_r(x))=\frac{1}{d-1}\gamma_d r^{d-1}+\mathcal{O}(r^{d+1}),\qquad (r\rightarrow0),
\end{equation}
and also the estimate
\begin{equation}\label{equ:capu}
\sd(C_r(x))\leq\frac{1}{d-1}\gamma_d r^{d-1},
\end{equation}
where \begin{equation}\label{equ:gammad}\gamma_d:=\frac{\Gamma(d/2)}{\Gamma((d-1)/2)\Gamma(1/2)}.\end{equation} Both estimates can be found in Section 3 of \cite{KS98}.

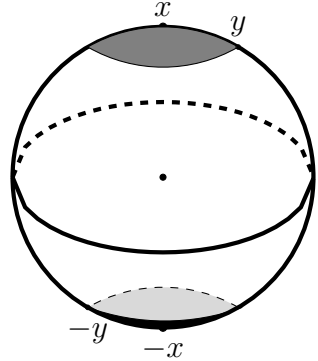
\begin{figure}[thb]
\begin{tikzpicture}
\fill[gray!30] (-1,-1.732) arc(240:300:2) (1,-1.732,0) arc(60:120:2);
\draw[ultra thick] (0,0) circle(2);
\draw[ultra thick, domain=-2:2] plot (\x, {-sqrt(1-\x*\x/4)});
\draw[ultra thick, dashed, domain=-2:2] plot (\x, {sqrt(1-\x*\x/4)});
\filldraw (0,0) circle(0.04);
\filldraw (0,2) circle(0.05) node[above] {$x$};
\filldraw (0,-2) circle(0.05) node[below] {$-x$};
\draw[fill=gray] (-1,1.732) arc(240:300:2) (1,1.732,0) arc(60:120:2);
\draw[dashed] (1,-1.732) arc(60:120:2);
\draw[fill=black] (-1,-1.732) arc(250:290:3) (1,-1.732,0) arc(-60:-120:2);
\draw (1,1.732) circle(0.03) node[above] {$y$};
\draw (-1,-1.732) circle(0.03) node[below] {$-y$};

\end{tikzpicture}
\caption{$F^{-1}(B(p_x,r))$}\label{fig:cap}
\end{figure}

\begin{lemma}\label{lem:upper}
When $\Hb=\R$, the uniform measure on $\Dc$  is  $(d-1)$-regular.
Moreover, we have the estimate  
\begin{equation}\label{equ:mupAp}
\Phi(\sd)(B(p_x,r))\leq \frac{2}{d-1}\gamma_dr^{d-1},\quad\text{ for any }p_x\in\Dc, 0<r\leq diam(\Dc)
\end{equation}
\end{lemma}
\begin{proof} 
$\Dc$ is the projective space embeded in $\R\Mb_{d\times d}^h$. $p_x$ and $p_y$ are furthest away if $x\perp y$, so $\text{diam}(\Dc)=\sqrt{2}$.

For any point $p_x=F(x)\in \Dc$ and any $r\leq\text{diam}(\Dc)=\sqrt{2}$, suppose in the set $B(p_x,r)\cap \Dc$, $p_y$ is the point that is furthest away from $p_x$. We can pick $y$ so that $\langle x,y\rangle\geq0$. Then $\Phi^{-1}(B(p_x,r))$ is the union of the spherical cap centered at $x$ with boundary point  $y$ together with its antipodal image, see Figure \ref{fig:cap}. 
By \eqref{equ:essential},
$$|\langle x,y\rangle|^2=1-\|p_x-p_y\|^2/2\geq1-r^2/2,$$
$$||x-y||^2=2-2\langle x,y\rangle\leq 2-2\sqrt{1-r^2/2}\leq2-2(1-r^2/2)=r^2.$$
By \eqref{equ:capu},
\begin{equation*}
\Phi(\sd)(B(p_x,r))=\sd(\Phi^{-1}(B(x,r)))=2\sd(C_r(x))\leq \frac{2}{d-1}\gamma_dr^{d-1}.
\end{equation*}
\end{proof}

\subsection{Expected value of coherence}
Let $X=\{x_i\}\in\Sc(d,N)$ be a random configuration on the sphere where each point is selected from a uniform distribution on the sphere. Let $\Theta=\min_{i\neq j}\arccos\langle x_i,x_j\rangle$, so
\begin{equation}\label{equ:ab1}
\xi(X)=\max_{i\neq j}|\langle x_i,x_j\rangle|\geq\max_{i\neq j}\langle x_i,x_j\rangle=\cos\Theta\geq1-\Theta^2/2.
\end{equation}
It is proven in \cite[Theorem 2]{CFJ13} that $F_N(t):=\Pr(N^{2/(d-1)}\Theta\leq t)\rightarrow F(t)$ where $F(t)=1-\exp(-\frac{\gamma_d}{2(d-1)}t^{d-1})$ is supported on $(0,\infty)$.

In order to compute the expected value of $\Theta^2$, we define
$G_N(s):=\Pr(N^{4/(d-1)}\Theta^2\leq s)=F_N(\sqrt{s})\rightarrow F(\sqrt{s})$. By a similar argument as the one in \cite[Corollary 3.4]{BRSSWW18}, we get
\begin{align*}
\lim_{N\rightarrow\infty}\E(N^{4/(d-1)}\Theta^2)&=\lim_{N\rightarrow\infty}\int_0^{\infty}(1-G_N(s))ds\\
&=\int_0^{\infty}1-F(\sqrt{s})ds=\int_0^{\infty}\exp(-\frac{1}{2}\kappa_{d}s^{\frac{d-1}{2}}):=C_d
\end{align*}

By \eqref{equ:ab1}, we have 
$$\E(\xi(X))\geq\E(1-\Theta^2/2)\sim1-\frac{C_d}{2}N^{-\frac{4}{d-1}}.$$
\bibliographystyle{amsplain}

\end{document}